\documentclass[11pt]{amsart}
\usepackage{amscd}
\usepackage[arrow,matrix]{xy}
\usepackage{graphicx, tikz}
\usepackage{comment}
\usepackage{amsmath}
\usepackage{etoolbox}
\usepackage{mathtools} 
\usepackage{latexsym, amssymb}
\numberwithin{equation}{section}
\theoremstyle{plain}

\newtheorem{lemma}{Lemma}[section]

\newtheorem{theorem}[lemma]{Theorem}

\theoremstyle{definition}

\newtheorem*{definition*}{Definition}
\newtheorem{remark}[lemma]{Remark}

\usepackage{color}
\usepackage{cancel}

\definecolor{brown}{RGB}{150,100,0}

\definecolor{purple}{RGB}{150,0,100}
\definecolor{grey}{RGB}{128,128,128}

\usepackage[hidelinks]{hyperref}
\newcommand{\R}{{\mathbb R}}

\newcommand{\E}{{\mathbb E}}

\newcommand{\Z}{{\mathbb Z}}

\newcommand{\Q}{{\mathbb Q}}

\newcommand{\Aa}{{\mathcal A}}

\newcommand{\Cc}{{\mathcal C}} 

\newcommand{\Ff}{{\mathcal F}}
\newcommand{\Gg}{{\mathcal G}} 
\newcommand{\Hh}{{\mathcal H}}

\newcommand{\Ll}{{\mathcal L}} 
\newcommand{\Mm}{{\mathcal M}} 

\newcommand{\Pp}{{\mathcal P}}

\newcommand{\Uu}{{\mathcal U}}

\newcommand{\Xx}{{\mathcal X}}
\newcommand{\Yy}{{\mathcal Y}}
\newcommand{\Zz}{{\mathcal Z}}

\mathchardef\mhyp="2D

\newcommand{\eps}{{\varepsilon}}

\newcommand{\pb}{\mathbf{p}}

\newcommand{\Meas}{{\rm Meas}}

\def\NABLA#1{{\mathop{\nabla\kern-.5ex\lower1ex\hbox{$#1$}}}}
\def\Nabla#1{\nabla\kern-.5ex{}_#1}

\newcommand{\p}{{\partial}}

\newcommand{\LRA}{\Longrightarrow}

\newcommand{\rom}[1]{{\em #1}}

\renewcommand{\)}{\rom)}

\DeclareMathOperator{\supp}{supp}

\begin{document}
\title[Estimating Conditional Expectation]{Estimating conditional expectation}
\author[H. V. L\^e]{H\^ong V\^an L\^e}	
\address{Institute of Mathematics of the Czech Academy of Sciences, Zitna 25, 11567 Praha 1, Czech Republic}
\email{hvle@math.cas.cz}

\date{August, 10, 2026}

\begin{abstract} 
In this paper, we consider the problem of estimating conditional expectations as an ill-posed inverse problem. We propose a solution based on a generalization of Vapnik's theorem \cite[Theorem 7.2]{Vapnik1998} for solving stochastic ill-posed problems in Hilbert spaces. As an application, we derive  a  new upper bound for sample  errors of  conditional expectation estimation. 
\end{abstract}

\dedicatory{Dedicated to  the memory  of Professor L\^e D\~ung Tr\'ang}
\maketitle

\section{Introduction}\label{sec:intro}
For a measurable space $\Xx$, denote by $\Sigma_\Xx$ the $\sigma$-algebra of $\Xx$ and by $\Pp(\Xx)$ the space of all probability measures on $\Xx$. The space $\Pp(\Xx)$ is equipped with the $\sigma$-algebra $\Sigma_w$, which is the smallest $\sigma$-algebra such that for any $A \in \Sigma_\Xx$, the evaluation map 
$$ev_{A}: \Pp(\Xx) \to \R, \quad \mu \mapsto \mu(A)$$ 
is measurable.

For measurable spaces $\Xx, \Yy$, denote by $\Meas(\Xx, \Yy)$ the space of all measurable mappings from $\Xx$ to $\Yy$.

In supervised learning, given a data set of labeled items
\[
S_m=\{(x_1,y_1),\ldots,(x_m,y_m)\}\in(\Xx\times\Yy)^m,
\]
sampled according to $\mu^m$, where $\mu\in \Pp(\Xx\times \Yy)$ is an unknown probability measure governing the distribution of i.i.d. labeled items $(x_i, y_i)\in (\Xx \times \Yy, \mu)$, $i \in [1, m]$, the aim of conditional probability estimation is to find the best approximation $f_{S_m}$ of a measurable map $\pb: \Xx \to \Pp(\Yy)$ in a hypothesis class $\Hh \subset \Meas(\Xx, \Pp(\Yy))$ which is a regular conditional probability measure for $\mu$ with respect to the projection $\Pi_\Xx: \Xx\times \Yy \to \Xx$.
 
In this paper, we assume there exists a regular conditional probability measure $\mu_{\Yy|\Xx} : \Xx \to \Pp(\Yy)$ for a joint probability distribution $\mu\in \Pp(\Xx\times \Yy)$ with respect to the projection $\Pi_\Xx:\Xx\times \Yy \to \Xx$, and we denote by $\mu_\Xx \coloneqq (\Pi_\Xx)_* \mu\in \Pp(\Xx)$ the marginal probability measure of $\mu$. For instance, if $\Yy$ is a Souslin space, there exists uniquely $\mu_\Xx$-a.e. a regular conditional probability measure $\Xx \to \Pp(\Yy)$ for $\mu$ with respect to $\Pi_\Xx$  \cite[Theorem 3.1 (4)]{LFR2004}. By Bayes' theorem \cite[Theorem 1.31, Problem 9]{Schervish1997}, there exists a regular conditional probability for $\mu$ with respect to the projection $\Pi_\Xx$ if there exists a regular conditional probability measure $\Yy \to \Pp(\Xx)$ for $\mu$ with respect to the projection $\Pi_\Yy: \Xx\times \Yy \to \Yy$, which, moreover, can be represented as a dominated Markov kernel.   
  
There are known algorithms for estimating regular conditional probability measures $\mu_{\Yy|\Xx}$ for $\mu$ with respect to the projection $\Pi_\Xx$ based on i.i.d. samples $S_m \in (\Xx\times \Yy)^m$, where $\mu$ is unknown, via Conditional Mean Embedding (CME) estimation \cite{PM2020}.
 
The problem of CME estimation  can be regarded as  Hilbert-space-valued  regression  estimation, see, e.g., \cite{PM2020}  and references therein. $\R$-valued  regression estimation is a particular  case of conditional expectation estimation, which we consider in this paper.  Denote by $\Ll^1(\Yy, \mu_\Yy)$ the space of all $\mu_\Yy$-integrable functions on $\Yy$. For a function $g \in \mathcal{L}^1(\mathcal{Y}, \mu_\mathcal{Y})$, it is known (e.g., using the disintegration formula) that
\begin{equation}
r_g^\mu(x) \coloneqq \int_{\mathcal{Y}} g(y) \, d\mu_{\mathcal{Y}\vert{}\mathcal{X}}(y \mid x) \in \Ll^1(\Xx, \mu_\Xx)\label{eq:rgmu}
\end{equation}
is a regular conditional expectation of $\Pi_\mathcal{Y}^*(g) \in \mathcal{L}^1(\mathcal{X} \times \mathcal{Y}, \mu)$. More precisely, $ \Pi_\Xx^*(r_g^\mu) \in \Ll^1(\Xx\times \Yy, \mu)$ is a regular version of the conditional expectation $\E_\mu(\Pi_\Yy^*(g)|\Pi_\Xx)$.

We propose a novel approach to estimating the regular conditional expectation $r_g^\mu$ via i.i.d. samples $S_m \in (\Xx\times \Yy)^m$ (where $g$ is known but $\mu$ is unknown) by   formulating this   task  as a stochastic ill-posed inverse problem. We solve this by applying a generalization of Vapnik's method for solving such problems.  We postpone applications of our method to CME  estimation to a later  paper.

This paper is organized as follows. In Section \ref{sec:vapnikh}, we provide a generalization of Vapnik's theorem \cite[Theorem 7.2]{Vapnik1998} for solving stochastic ill-posed problems in Hilbert spaces and discuss    related  results (Theorem \ref{thm:cvapnikhl},  Remark \ref{rem:cvapnikhl}). In Section \ref{sec:estce}, we apply  Theorem \ref{thm:cvapnikhl} to the problem of estimating $r_g^\mu$ via the sample $S_m$ where $\mu$ is unknown  and discuss  related  results (Theorems \ref{thm:genvapnik}, \ref{thm:uniconsistency}, Remarks \ref{rem:VI}, \ref{rem:genvapnik}).  This paper also contains Appendix  \ref{sec:HK}  where we give a proof of a technical lemma \ref{lem:HK}, using generalized 
Koksma--Hlawka inequality of Aistleitner and Dick \cite{AD2015}.

\section*{Acknowledgement}
Research of HVL was supported by the Institute of Mathematics, Czech Academy of Sciences (RVO: 67985840). 

\section{A generalization of Vapnik's theorem for Hilbert spaces}\label{sec:vapnikh}

\begin{theorem}[A generalization of Vapnik's theorem]\label{thm:cvapnikhl}  Let $E_1$ be a real Hilbert space and $E_2$ a normed vector space with metric
	\[
	\rho_2(u,v)\coloneqq\|u-v\|_{E_2}.
	\]
	Let $A\in\mathcal L(E_1,E_2)$, let $\mathcal M\subset E_1$ be
	nonempty, norm closed, and convex, and suppose that the restriction
	$A|_{\mathcal M}$ is injective. Let $f\in\mathcal M$ and set $F=Af$.
	
	For every $m\in\mathbb N^+$, let
	$(\Xx_m,\mu_m)$ be a probability space. For each
	$S_m\in\Xx_m$, let $F_{S_m}\in E_2$ and
	$A_{S_m}\in\mathcal L(E_1,E_2)$.
For $\gamma_m>0$, let $R_{\gamma_m} (\cdot, F_{S_m}, A_{S_m}): \Mm \to \R_{\ge 0}$ be defined by
\[
R_{\gamma_m}(h,F_{S_m},A_{S_m}) \coloneqq \|A_{S_m}h-F_{S_m}\|_{E_2}^{2} + \gamma_m\|h\|_{E_1}^{2}.
\]

1) Then $R_{\gamma_m} (\cdot, F_{S_m}, A_{S_m})$ has a unique minimizer $f_{S_m} \in \Mm$.

Fix $\varepsilon>0$, and choose $C_1 > 0$ and $C_2\ge 0$ such that
\[
\bigl(C_1+C_2\|f\|\bigr)^2< \eps.
\]
2) Then there exists
\[
\gamma_0=\gamma_0(\varepsilon,C_1,C_2,f,A,\mathcal M)>0
\]
such that, whenever $\gamma_m\le\gamma_0$,
\begin{equation}
	\begin{aligned}
		\mu_m^*\bigl\{S_m:\|f_{S_m}-f\|^2>\varepsilon\bigr\}
		&\le
		\mu_m^*\bigl\{S_m:\|F_{S_m}-F\|>C_1\sqrt{\gamma_m}\bigr\}\\
		&\quad+
		\mu_m^*\bigl\{S_m:
		\|A_{S_m}-A\|_{\mathrm{op}}>C_2\sqrt{\gamma_m}
		\bigr\}.\label{eq:cpe}
	\end{aligned}
\end{equation}
\end{theorem}
Here  $\mu_m^*$ denotes the outer measure defined by $\mu_m$. The use  of the outer measure is necessary since   the map $S_m \mapsto f_{S_m}$ is not  required  to be measurable. 
\begin{proof}[Proof of Theorem \ref{thm:cvapnikhl}] 
 1)   Fix $S_m\in\Xx_m$. Since
\[
R_{\gamma_m}(h,F_{S_m},A_{S_m})
\ge \gamma_m\|h\|_{E_1}^2,
\]
the functional $R_{\gamma_m}(\cdot,F_{S_m},A_{S_m})$ is coercive
on $\mathcal M$. Let $(h_j)\subset\mathcal M$ be a minimizing
sequence. Then $(h_j)$ is bounded. Since $E_1$ is reflexive, after
passing to a subsequence we have $h_j\rightharpoonup h$ in $E_1$.
Because $\mathcal M$ is norm closed and convex, it is weakly closed,
so $h\in\mathcal M$.

The map
\[
h\longmapsto
\|A_{S_m}h-F_{S_m}\|_{E_2}^2+\gamma_m\|h\|_{E_1}^2
\]
is weakly lower semicontinuous. Hence $h$ is a minimizer. Finally,
the first term is convex and the second is strictly convex, so the
minimizer is unique. We denote it by $f_{S_m}$.

2) Consider the good event
\begin{equation}
\mathcal C_m \coloneqq
\left\{ S_m \in \Xx_m: \|F_{S_m}-F\|\le C_1\sqrt{\gamma_m}, \quad \|A_{S_m}-A\|_{\mathrm{op}}\le C_2\sqrt{\gamma_m} \right\}.\label{eq:cm}
\end{equation}

Since $f_{S_m}$ is the minimizer of $R_{\gamma_m} (\cdot, F_{S_m}, A_{S_m})$, we have
\begin{equation}
\|A_{S_m}f_{S_m}-F_{S_m}\|^2+\gamma_m\|f_{S_m}\|^2 \le \|A_{S_m}f-F_{S_m}\|^2+\gamma_m\|f\|^2.\label{eq:est1}
\end{equation}
For $S_m \in \mathcal C_m$, by the triangle inequality, we have
\begin{equation}
\|A_{S_m}f-F_{S_m}\| \le \|(A_{S_m}-A)f\|+\|F-F_{S_m}\| \stackrel{\eqref{eq:cm}}{\le} \sqrt{\gamma_m}\bigl(C_2\|f\|+C_1\bigr).\label{eq:est2}
\end{equation}
Set $$r \coloneqq C_1+C_2\|f\|.$$
We    obtain from  \eqref{eq:est1},  dropping the  first nonnegative  term on  its LHS,  taking into account \eqref{eq:est2}: 
\begin{equation}
\|f_{S_m}\|^2\le \|f\|^2+r^2  \qquad\text{ for }  S_m \in   \Cc_m.\label{eq:fmbounded}
\end{equation}
From \eqref{eq:cm} and \eqref{eq:fmbounded}, we obtain 
\begin{equation}
\{ f_{S_m}: S_m \in \Cc_m\} \subset K \coloneqq \left\{ h \in \mathcal M: \, \|h\|^2\le\|f\|^2+r^2 \right\}. \label{eq:fmbounded1}
 \end{equation}

The set $K$ is weakly compact. Indeed, $\mathcal M$ is weakly
closed because it is norm closed and convex, while the closed ball
of $E_1$ appearing in the definition of $K$ is weakly compact.

Since $f_{S_m}$ is the minimizer of $R_{\gamma_m} (\cdot, F_{S_m}, A_{S_m})$, combining \eqref{eq:est1} and \eqref{eq:est2}, for $S_m \in \Cc_m$ we have
\begin{equation}
\|A_{S_m}f_{S_m}-F_{S_m}\| \le \sqrt{\gamma_m}\sqrt{\|f\|^2+r^2}. \label{eq:dist21}
\end{equation}
Set $K_0\coloneqq C_1+(1+C_2)\sqrt{\|f\|^2+r^2}$. 

Then for $S_m \in \Cc_m$, we have
\begin{equation}
\begin{aligned}
	\|Af_{S_m}-F\|
	&\le
	\|(A-A_{S_m})f_{S_m}\| +\|A_{S_m}f_{S_m}-F_{S_m}\| +\|F_{S_m}-F\|\\
	&\stackrel{\eqref{eq:dist21}}{\le} K_0\sqrt{\gamma_m}.\label{eq:dist22}
\end{aligned}
\end{equation}

\begin{lemma}\label{lem:weakest}
For every $g\in E_1$ and every $\eta>0$, there exists $\delta>0$ such that
\[
h\in K,\qquad \|Ah-Af\|<\delta \quad\Longrightarrow\quad |\langle h-f,g\rangle|<\eta.
\]
\end{lemma}
\begin{proof}
	Suppose the assertion is false. Then, for every $j\in\mathbb N^+$,
	there exists $h_j\in K$ such that
	\[
	\|Ah_j-Af\|<\frac1j,
	\qquad
	|\langle h_j-f,g\rangle|\ge\eta.
	\]
	Since $K$ is weakly compact, the Eberlein--Shmulyan theorem \cite[p. 141]{Yosida1995} implies
	that, after passing to a subsequence,
	\[
	h_j\rightharpoonup h
	\qquad\text{for some }h\in K.
	\]
	Since $A$ is bounded and linear, $Ah_j\rightharpoonup Ah$ in $E_2$.
	On the other hand, $Ah_j\to Af$ in norm, and hence weakly.
	Consequently, $Ah=Af$. Since $h,f\in\mathcal M$ and
	$A|_{\mathcal M}$ is injective, we obtain $h=f$. But then
	\[
	\langle h_j-f,g\rangle\longrightarrow0,
	\]
	contradicting
	$|\langle h_j-f,g\rangle|\ge\eta$.
\end{proof}

{\it Continuation of the proof of Theorem \ref{thm:cvapnikhl}.}  Apply Lemma \ref{lem:weakest} with 
$$g=f  \text{ and  }\eta=\frac{\varepsilon-r^2}{4},$$ 
and let $\delta>0$ be the resulting constant.
Choose $\gamma_0>0$ so small that
\[
K_0\sqrt{\gamma_0}<\delta.
\]
Then, whenever $\gamma_m\le\gamma_0$, \eqref{eq:dist22} and
Lemma \ref{lem:weakest} give
\begin{equation}
|\langle f_{S_m}-f,f\rangle|
<\frac{\varepsilon- r^2}{4},
\qquad S_m\in\mathcal C_m.\label{eq:dist1m}
\end{equation}

At the same time, by \eqref{eq:fmbounded1}:
\begin{equation}
\|f_{S_m}\|^2-\|f\|^2 \le r^2.\label{eq:cmf}
\end{equation}
Using
\begin{equation*}
\|f_{S_m}-f\|^2 =\|f_{S_m}\|^2-\|f\|^2 + 2\langle f,f-f_{S_m}\rangle,
\end{equation*}
and taking into account \eqref{eq:dist1m} and \eqref{eq:cmf}, we obtain
\begin{equation}
\|f_{S_m}-f\|^2\le  r^2  + \frac{\eps - r^2}{2} = \frac{\eps  + r^2}{2} < \eps \qquad \text{ for } S_m \in \mathcal C_m.\label{eq:cpeb}
\end{equation}
Taking into account the subadditivity of outer measure, we derive \eqref{eq:cpe} from \eqref{eq:cpeb} immediately. This completes the proof of Theorem \ref{thm:cvapnikhl}. 
\end{proof}

\begin{remark}\label{rem:cvapnikhl}   
(1) When $A_{S_m}=A$ for $S_m \in \Xx_m$, take $C_2=0$ and
\[
C_1\coloneqq\sqrt{\eps/2}.
\]
Then, for sufficiently small $\gamma_m$,
\begin{equation}
	\mu_m^*\{\|f_{S_m}-f\|^2>\varepsilon\}
	\le
	\mu_m^*
	\left\{
	\|F_{S_m}-F\|>
	\sqrt{\frac{\varepsilon\gamma_m}{2}}
	\right\}.
	\label{eq:vapnik2b}
\end{equation}
Inequality \eqref{eq:vapnik2b} improves the numerical constant in
\cite[Theorem 7.2, p.~298]{Vapnik1998} (for the case $\Mm = E_1$ and $A_{S_m} = A$): its squared threshold is
twice the squared threshold appearing there.

(2) In \cite[Theorem 7.3, p. 299]{Vapnik1998}, Vapnik also considers a method for solving the stochastic ill-posed problem $Af = F$ via a family of ``empirical equations" $A_{S_m} f = F_{S_m}$, using a regularizer $W$ whose sublevel sets $W^{-1}([0, c])$ are compact for all $c \in \mathbb{R}_{\ge 0}$.

(3)   The linearity of the forward operator is not essential for the
compactness argument underlying Theorem \ref{thm:cvapnikhl}.
An analogous result holds for injective weakly sequentially continuous
maps $A:\mathcal M\to E_2$, where weak convergence in $E_1$ implies
metric convergence in $E_2$, provided that the corresponding
regularized functionals admit minimizers. Since all empirical integral
 operators considered below are linear, we omit the
nonlinear formulation.
\end{remark}

\section{Estimating conditional expectation with generalized Vapnik's theorem}\label{sec:estce}

Let $\mathcal{X}$ be a compact smooth Riemannian submanifold in $\mathbb{R}^d$, equipped with the induced Riemannian metric, and let $\mathcal{Y}$ be a measurable space. In this section, we apply Theorem \ref{thm:cvapnikhl} to estimate the regular conditional expectation $r_g^\mu \in \Ll^1(\Xx, \mu_\Xx)$ defined by \eqref{eq:rgmu} via i.i.d. samples $S_m \in \big((\Xx\times \Yy)^m,\mu^m\big)$ where $g$ is known but $\mu$ is unknown.

To estimate the regular conditional expectation $r_g^\mu$, we must restrict our search to a well-behaved hypothesis space of approximations in $\Ll^1(\Xx, \mu_\Xx)$. Let $K: \Xx \times \Xx \to \R$ be a smooth Mercer kernel. We denote by $\Hh_K$ the Reproducing Kernel Hilbert Space (RKHS) associated with $K$. Because $\Xx$ is a compact smooth Riemannian submanifold and $K$ is a smooth kernel, the space $\Hh_K$ consists entirely of smooth functions (see the proof of Theorem D in \cite{CS2001}). Furthermore, $\Hh_K$ embeds continuously into the space of continuous functions  $C(\Xx)$ endowed  with the sup norm  $\|\cdot \|_\infty$. Specifically, the inclusion operator $I_K: \Hh_K \to C(\Xx)_\infty$ is bounded, with its operator norm satisfying \cite[Chapter II]{CS2001}:
\begin{equation}
\Vert I_K\Vert \le \sqrt{C_K}, \qquad C_K \coloneqq \max_{x \in \Xx} K(x, x).\label{eq:ikbounded}
\end{equation}
In the proof of  \cite[Theorem D]{CS2001}  Cucker and Smale show that  
$ I_K =  J_s \circ J^s$ where  $J^s:  \Hh_K \to H^s(\Xx)$ is a continuous   embedding and $J_s:   H^s (\Xx) \to  C(\Xx)_\infty$ is a  compact embedding if $s > \frac{\dim \Xx}{2}$.  By  the Rellich-Kondrachov Theorem for compact manifold with boundary, see, e.g.,  Gilbarg-Trudinger \cite[Theorem 7.26, p. 171]{GT2001}
the embedding $H^s(\Xx) \to C^r(\Xx)$ is compact if $s - \frac{\dim \Xx}{2} > r$.  Thus, choosing $s>r+\dim(\Xx)/2$, we conclude that the
embedding $\Hh_K\hookrightarrow C^r(\Xx)$ is compact.	In particular, for all nonnegative integers $r$ there exists a constant $c_r (\Xx, K)$ such that for any $f \in \Hh_K$ we have, 
\begin{equation}\label{eq:lipsmooth}
	\|f\|_{C^r (\Xx)} \le c_r (\Xx, K)\|f\|_{\Hh_K}. 
\end{equation}

To formulate the estimation of $r_g^\mu$ as a stochastic ill-posed problem without relying on the cumulative distribution function of the joint measure, we utilize the multivariable relative indicator function $I^{(d)}: \R^d \times \R^d \to [0,1]$ defined by:
\begin{equation}
	I^{(d)}(a, x) \coloneqq \prod_{i=1}^d 1_{[0, \infty)}(a_i - x_i).\label{eq:rif}
\end{equation}

\begin{lemma}\label{lem:regillp}
	Let $\Xx\subset\mathbb R^d$ be compact, and let
	$Q\subset\mathbb R^d$ be a compact rectangle such that
	\[
	\Xx\subset\operatorname{int}(Q).
	\]
	Let $\nu$ be the normalized Lebesgue measure on $Q$, and let
	$\mu_\Xx\in\mathcal P(\Xx)$. Define
	\[
	A_{\mu_\Xx}:C(\Xx)_\infty\longrightarrow L^2(Q,\nu)
	\]
	by
	\[
	(A_{\mu_\Xx}f)(a)
	\coloneqq
	\int_\Xx I^{(d)}(a,x)f(x)\,d\mu_\Xx(x),
	\qquad a\in Q.
	\]
	
	\begin{enumerate}
		\item The operator $A_{\mu_\Xx}$ is bounded and
		\[
		\|A_{\mu_\Xx}\|_{\mathrm{op}}\le1.
		\]
		
		\item The operator $A_{\mu_\Xx}:C(\Xx)_\infty\to L^2(Q,\nu)$
		is compact.
		
		\item If, moreover, $\supp(\mu_\Xx)=\Xx$, then
		$A_{\mu_\Xx}$ is injective.
	\end{enumerate}
\end{lemma}

\begin{proof}
	(1)	For $f\in C(\mathcal X)$ and $a\in Q$,
	\[
	|(A_{\mu_{\Xx}}f)(a)|
	\le
	\int_{\Xx}|f(x)|\,d\mu_{\Xx}(x)
	\le \|f\|_\infty.
	\]
	Since $\nu(Q)=1$, it follows that
	\[
	\|A_{\mu_{\Xx}}f\|_{L^2(Q,\nu)}
	\le \|f\|_\infty.
	\]
	This proves Assertion (1).
	
	(2)	For Assertion (2), consider the bounded inclusion
	\[
	J:C(\Xx)_\infty\longrightarrow
	L^2(\Xx,\mu_{\Xx})
	\]
	and the integral operator
	\[
	T:L^2(\Xx,\mu_{\Xx})
	\longrightarrow L^2(Q,\nu),
	\qquad
	(Tg)(a)
	=
	\int_{\Xx} I^{(d)}(a,x)g(x)\,
	d\mu_{\mathcal X}(x).
	\]
	Because the kernel $I^{(d)}$ is bounded,
	\[
	I^{(d)}\in
	L^2(Q\times\Xx,\nu\otimes\mu_{\Xx}),
	\]
	so $T$ is a Hilbert--Schmidt operator and hence compact.
	Since
	\[
	A_{\mu_{\Xx}}=T\circ J,
	\]
	the operator $A_{\mu_{\Xx}}$ is compact.
	
	(3)	For Assertion (3), suppose that
	\begin{equation}
		A_{\mu_{\Xx}}f=0
		\qquad\text{in }L^2(Q,\nu).\label{eq:vanishA}
	\end{equation}
	Define the finite signed Borel measure $\lambda_f$ on $\R^d$ by
	\[
	\lambda_f(B)
	\coloneqq
	\int_{B\cap\Xx}f(x)\,d\mu_{\Xx}(x).
	\]
	Its distribution function is
	\begin{equation}
		F_{\lambda_f}(a)
		=
		\lambda_f\big((-\infty,a]\big)
		=
		(A_{\mu_{\Xx}}f)(a).\label{eq:cdafa}
	\end{equation}

	Since $\nu$ is equivalent to the  restriction of the Lebesgue measure to $Q$,  \eqref{eq:vanishA}  and \eqref{eq:cdafa} imply
	\[
	 F_{\lambda_f}=0
	\qquad\text{Lebesgue-almost everywhere on }Q.
	\]
	In the sense of distributions on $\operatorname{int}Q$,
	\[
	\partial_1\cdots\partial_d  F_{\lambda_f}=\lambda_f.
	\]
	Consequently,
	\[
	\lambda_f|_{\operatorname{int}Q}=0.
	\]
	But $\lambda_f$ is supported on
	$\mathcal X\subset\operatorname{int}Q$, and therefore
	$\lambda_f=0$. Thus $f=0$ $\mu_{\mathcal X}$-almost everywhere.
	
	Finally, because $f$ is continuous and
	$\operatorname{supp}(\mu_{\mathcal X})=\mathcal X$, this implies
	$f=0$ everywhere on $\mathcal X$. Hence
	$A_{\mu_{\mathcal X}}$ is injective.
\end{proof}

By applying the disintegration theorem and Fubini's theorem to the true conditional expectation $r_g^\mu$ defined by \eqref{eq:rgmu}, we obtain:
\begin{align}
(A_{\mu_\mathcal{X}} r_g^\mu)(a) & = \int_{\mathcal{X}} I^{(d)}(a, x) \left( \int_{\mathcal{Y}} g(y) \, d\mu_{\mathcal{Y}\vert{}\mathcal{X}}(y \mid x) \right) d\mu_\mathcal{X}(x)\nonumber\\
& = \int_{\mathcal{X} \times \mathcal{Y}} I^{(d)}(a, x) g(y) \, d\mu(x, y).
\label{eq:amur}
\end{align}
We define the right-hand side of \eqref{eq:amur} as a function of the joint measure $\mu$:
\begin{equation}
B_\mu^g(a) \coloneqq \int_{\mathcal{X} \times \mathcal{Y}} I^{(d)}(a, x) g(y) \, d\mu(x, y), \qquad   a \in Q.\label{eq:bgmu}
\end{equation}
Consequently, the target regression function $r_g^\mu$ is the exact solution to the following operator equation:
\begin{equation}
A_{\mu_\mathcal{X}} r_g^\mu = B_\mu^g.\label{eq:feq}
\end{equation}

\begin{lemma}\label{lem:gl2}  Assume that  $g \in L^2 (\Yy, \mu_\Yy)$. Then  $B_\mu ^g \in  L^2  (Q, \nu)$.
\end{lemma}
\begin{proof}
 Since $0 \le I^{(d)}(a, x) \le 1$, we have:
\begin{equation}
\vert B_\mu^g(a) \vert \le \int_{\mathcal{X} \times \mathcal{Y}} \vert I^{(d)}(a, x) \vert \vert g(y)\vert \, d\mu(x, y) \le \int_{\mathcal{X} \times \mathcal{Y}} \vert g(y)\vert \, d\mu(x, y). \label{eq:bgbounded}
\end{equation}
Because $\mu_\Yy \in \Pp (\Yy)$, Hölder's inequality  guarantees that $L^2(\mathcal{Y}, \mu_\mathcal{Y}) \subset L^1(\mathcal{Y}, \mu_\mathcal{Y})$. Taking into account \eqref{eq:bgbounded},
we obtain
\begin{equation}
	 \vert B_\mu^g(a) \vert \le \Vert g \Vert_{L^1(\mu_\mathcal{Y})} \le \Vert g \Vert_{L^2(\mu_\mathcal{Y})} < \infty. \label{eq:bgboundedl2}
\end{equation}
From \eqref{eq:bgboundedl2} we conclude that $B_\mu^g \in L^2(Q, \nu)$.
\end{proof}

 Because the operator $A_{\mu_\mathcal{X}}$ is compact on $\mathcal{H}_K$, the equation $A_{\mu_\mathcal{X}} r = B_\mu^g$ is ill-posed if $\Hh_K$ is infinite dimensional. To solve  this equation, using empirical data $S_m \in (\Xx\times \Yy)^m$, we  assume that $r_g^\mu \in \mathcal{H}_K$ and apply Theorem \ref{thm:cvapnikhl}.  

Given an i.i.d. sample $S_m = \{(x_1, y_1), \dots, (x_m, y_m)\} \in (\Xx \times \Yy)^m$,  we set  for $a \in Q$
\begin{equation}
(A_{S_m} f)(a) \coloneqq \frac{1}{m} \sum_{i=1}^m I^{(d)}(a, x_i) f(x_i), \label{eq:empA}
\end{equation}
\begin{equation}
B_{S_m}^g(a) \coloneqq \frac{1}{m} \sum_{i=1}^m I^{(d)}(a, x_i) g(y_i). \label{eq:empB}
\end{equation}

We consider the statistical learning model $(\Zz, \mathcal{H}_K, R, \Pp_K(\Zz))$, where $\Zz \coloneqq \Xx \times \Yy$, $\Pp_K(\Zz)$ consists of all $\mu \in \Pp (\Zz)$  satisfying  the conditions (i) and (ii)  below:\\
(i) $\text{supp}(\mu_\mathcal{X}) = \mathcal{X}$;\\
(ii) The conditional expectation  admits a version $r_g^\mu \in \mathcal{H}_K$; \\
  and  the expected loss function $R: \mathcal{H}_K \times \Pp_K (\Zz) \to \mathbb{R}_{\ge 0}$ is defined as:
\begin{equation}
	R(f, \mu) \coloneqq \Vert A_{\mu_\mathcal{X}} f - B_\mu^g\Vert_{L^2(Q, \nu)}^2. \label{eq:truerisk}
\end{equation}

Given a regularization parameter $\gamma_m > 0$, we define the regularized empirical risk function $R_{\gamma_m}(\cdot, B^g_{S_m}, A_{S_m}): \mathcal{H}_K \to \mathbb{R}$ by:
\begin{equation}
	R_{\gamma_m}(f, B_{S_m}^g, A_{S_m}) \coloneqq \Vert A_{S_m} f - B_{S_m}^g\Vert_{L^2(Q, \nu)}^2 + \gamma_m \Vert f \Vert_{\mathcal{H}_K}^2. \label{eq:regrisk}
\end{equation}
Let $f_{S_m}$ be a minimizer of this risk. By the lower semicontinuity of the regularizer and the properties of $\mathcal{H}_K$, such a minimizer exists uniquely. We now state the generalized Vapnik theorem for the consistency of this learning algorithm.  

\begin{theorem}\label{thm:genvapnik}
	Assume the conditions of the statistical learning model $(\mathcal{Z}, \mathcal{H}_K,\\ R, \mathcal{P}_K(\mathcal{Z}))$ above. Let  $\sup_{\mu \in \mathcal{P}_K(\mathcal{Z})} \mathbb{E}_{\mu_\mathcal{Y}}[|g|^2] \le M < \infty$. 
	
	Let $\{\gamma_m \in \mathbb{R}^{+}, m \in \mathbb{N}^+\}$ be a sequence of regularization parameters,  such that there  exists a  sequence  of  $\beta_m \in \R^+$ with the following properties 
	\begin{equation}
		\lim_{m \to \infty} \gamma_m = 0, \qquad \lim _{m \to \infty}\beta_m\sqrt{\gamma_m}  =\infty,  \quad \lim_{m \to \infty} \frac{m\gamma_m}{\beta^2_m\log m} = +\infty. \label{eq:gamma_rate}
	\end{equation}
	Then for any $\mu \in \mathcal{P}_K(\mathcal{Z})$ and $\eps > 0$, 
	\begin{equation}
	\lim_{m\to \infty}	(\mu^m)^*\big\{ S_m \in \mathcal{Z}^m :\|f_{S_m} - r^\mu_g\|_{\Hh_K}\ge \eps \big\} = 0. \label{eq:vapnik_bound}
	\end{equation}
\end{theorem}

\begin{proof}
	The proof proceeds in four steps: verifying the prerequisites of Theorem \ref{thm:cvapnikhl}, applying its deviation bound \eqref{eq:cpe}, handling the unbounded target $g$ via truncation, and establishing  the consistency \eqref{eq:vapnik_bound}. 
	
	\underline{Step 1.} {\it Verification of Theorem \ref{thm:cvapnikhl} Conditions.}
	 Let $E_1 \coloneqq \mathcal{H}_K$, which is a Hilbert space, and $E_2 \coloneqq L^2(Q, \nu)$, which is a normed vector space. We define $\mathcal{M} \coloneqq \mathcal{H}_K$, which is trivially nonempty, norm-closed, and convex. 
	
	The operator $A \coloneqq A_{\mu_\mathcal{X}}: \mathcal{H}_K \to L^2(Q, \nu)$ is bounded. To see this, note that by \eqref{eq:ikbounded} the inclusion $I_K: \mathcal{H}_K \hookrightarrow C(\mathcal{X})_\infty$ is bounded (with norm $\le \sqrt{C_K}$), and by Lemma \ref{lem:regillp}(1), $A_{\mu_\mathcal{X}}: C(\mathcal{X}) \to L^2(Q, \nu)$ is continuous with operator norm less than or  equal to 1. Hence, $A_{\mu_\mathcal{X}} \in \mathcal{L}(E_1, E_2)$. By Lemma \ref{lem:regillp}(3), because $\text{supp}(\mu_\mathcal{X}) = \mathcal{X}$, the restriction of $A_{\mu_\mathcal{X}}$ to $\mathcal{M}$ is injective. Finally   we let $(\Xx_m, \mu_m) \coloneqq (\Zz^m, \mu^m)$. For every $S_m\in\Zz^m$, Lemma \ref{lem:regillp}(1), applied
	to the empirical marginal
	\[
	\mu_{S_m,\Xx}=\frac1m\sum_{i=1}^m\delta_{x_i},
	\]
	together with \eqref{eq:ikbounded}, shows that
	$A_{S_m}\in\mathcal L(\Hh_K,L^2(Q,\nu))$ and
	$\|A_{S_m}\|_{\mathrm{op}}\le\sqrt{C_K}$.  Since  for any $a \in Q$
	$$|B^g _{S_m}  (a) | \le \frac{1}{m}\sum_{i=1}^m |g(y_i)| <\infty $$
	we conclude that  $B^g_{S_m}\in  L^2 (Q,\nu)$
	
	Setting the target $F \coloneqq B_\mu^g$ and noting that the true minimizer $f \coloneqq r_g^\mu \in \mathcal{M}$ satisfies $A_{\mu_\mathcal{X}} r_g^\mu = B_\mu^g$, all  structural conditions of Theorem \ref{thm:cvapnikhl} are satisfied.
	
	\underline{Step 2.} {\it Application of the  deviation  bound \eqref{eq:cpe}.}  Choose $C_1,C_2>0$ such that
	\[
	(C_1+C_2\|r_g^\mu\|_{\Hh_K})^2<\eps.
	\]
	By  Theorem \ref{thm:cvapnikhl}  there exists $\gamma_0 > 0$ such that for any $\gamma_m \le \gamma_0$ we have
	\begin{align}
		(\mu^m)^*\big( \Vert f_{S_m} - r_g^\mu \Vert_{\mathcal{H}_K}^2& > \eps \big) \le (\mu^m)^*\big( \Vert B_{S_m}^g - B_\mu^g \Vert_{L^2} > C_1 \sqrt{\gamma_m} \big)\nonumber\\
		& + (\mu^m)^*\big( \Vert A_{S_m} - A_{\mu_\mathcal{X}} \Vert_{\text{op}} > C_2 \sqrt{\gamma_m} \big), \label{eq:master_decomp}
	\end{align}
	where $(C_1 + C_2 \Vert r_g^\mu \Vert_{\mathcal{H}_K})^2 < \eps$.
	
	\underline{Step 3.} {\it Bounding  the  RHS of \eqref{eq:master_decomp} with truncation.}
	Recall that
	$$\Vert A_{S_m} - A_{\mu_\mathcal{X}} \Vert_{\text{op}} = \sup_{\Vert h \Vert_{\mathcal{H}_K} \le 1} \Vert A_{S_m} h - A_{\mu_\mathcal{X}} h \Vert_{L^2(Q, \nu)}.$$
	 
	Because $\nu \in \Pp (Q)$ (and thus $\nu(Q) = 1$),  we have  
	\begin{align}\Vert A_{S_m} h - A_{\mu_\mathcal{X}} h \Vert_{L^2(Q, \nu)} &\le \sup_{a \in Q} \vert{} (A_{S_m} h)(a) - (A_{\mu_\mathcal{X}} h)(a) \vert{}\nonumber \\ 
	&= \sup_{a \in Q} \left\vert{} \int_{\mathcal{X}} I^{(d)}(a, x) h(x) \, d(\mu_{S_m^\Xx} - \mu_\Xx)(x) \right\vert{}.\label{eq:bound1}
	\end{align}

Here  for $S_m = \big((x_1, y_1), \ldots,  (x_m, y_m)\big)$  we denote by  $S_m^\Xx\coloneqq  (x_1, \ldots, x_m)$   the  $\Xx$-component of $S_m$.

\begin{lemma}\label{lem:HK} There  exists  a  constant  $C (K, Q) $  depending only on   the kernel $K$, the submanifold $\Xx$, and the compact rectangle $Q\supset \Xx$ (we suppress the dependence on $\Xx$ in the notation)  such that  for all  $h \in \Hh_K$ with $\| h \|_{\Hh_K}\le 1 $  and every $a \in Q$ we have
	\begin{equation}
	\left\vert{} \int_{\Xx} I^{(d)}(a,\cdot) h \, d(\mu_{S_m^\Xx} - \mu_\Xx) \right\vert{} \le C(K, Q) \sup_{x \in Q} \left\vert{} (\E_{\mu_{S_m^\Xx}} - \E_{\mu_\Xx})\big(I^{(d)}(x, \cdot)\big) \right\vert{}.\label{eq:HK}
\end{equation}
\end{lemma}
 We postpone the proof of Lemma  \ref{lem:HK} to Appendix \ref{sec:HK}. 
	 Taking into account \eqref{eq:bound1}, we obtain  from \eqref{eq:HK}
	 \begin{equation}
	 \Vert A_{S_m} - A_{\mu_\Xx} \Vert_{\text{op}} \le C(K, Q) \sup_{f \in \mathcal{F}} \vert{}\E_{\mu_{S_m^\Xx}} f - \E_{\mu_\Xx} f\vert{},\label{eq:bound2}
	 \end{equation}
	 where 
	 $$\mathcal{F} \coloneqq \{I^{(d)}(a, \cdot) : a \in Q\}\subset \{0, 1\}^\Xx.$$
	 
	 From  \eqref{eq:bound2} we obtain immediately
	 \begin{align}(\mu^m) ^*\left(S_m\in \Zz^m:  \Vert A_{S_m} - A_{\mu_\mathcal{X}} \Vert_{\text{op}} > C_2 \sqrt{\gamma_m} \right)\nonumber\\  
	  \le (\mu^m_\Xx)^*\left(S_m\in \Xx^m: \sup_{f \in \mathcal{F}} \vert{}\E_{\mu_{S_m}} f - \E_{\mu_\Xx} f\vert{} > \frac{C_2 \sqrt{\gamma_m}}{C(K, Q)} \right).\label{eq:bound3a}
	 \end{align}
	 For notational simplicity, let $\epsilon_m \coloneqq \frac{C_2 \sqrt{\gamma_m}}{C(K, Q)}$.

	 A class $\Hh$ of functions on $\Xx$ is called \emph{sequentially
	 	pointwise separable}  if there exists a countable subclass
	 $\Hh_0\subset\Hh$ such that, for every $h\in\Hh$, there is a sequence
	 $(h_k)\subset\Hh_0$ satisfying \footnote{
	 	This property is called pointwise measurability in
	 	\cite[p. 110]{VW1996}.}
	 \[
	 h_k(x)\longrightarrow h(x)
	 \qquad\text{for every }x\in\Xx.
	 \]

	 \begin{lemma}\label{lem:pm_measurability}
	 	The function class $\mathcal{F} \coloneqq \{I^{(d)}(a, \cdot) : a \in Q\}$ is  sequentially pointwise separable. Consequently, the mapping 
	 	$$ (S_{2m}, \sigma) \mapsto \sup_{f \in \mathcal{F}} \left\vert \frac{1}{m} \sum_{i=1}^m \sigma_i (f(x_i) - f(x_i')) \right\vert $$
	 	is measurable with respect to the product $\sigma$-algebra of $\Xx^{2m} \times \{ \pm 1\}^m$.
	 \end{lemma}
	 \begin{proof}[Proof of Lemma \ref{lem:pm_measurability}] 
	 Write $Q = \prod_{j=1}^d [\alpha_j, \beta_j]$. For any $a \in Q$, we construct a sequence $q_k \in \Q^d \cap Q$ as follows. Since $\Xx$ is compact and $\Xx \subset \operatorname{int}(Q)$, we have $s_j \coloneqq \sup_{x \in \Xx} x_j < \beta_j$ for every $j$. If $a_j = \beta_j$, choose any rational $q_{k, j} \in (s_j, \beta_j)$, so that $1_{[0, \infty)}(q_{k, j} - x_j) = 1 = 1_{[0, \infty)}(a_j - x_j)$ for all $x \in \Xx$; if $a_j < \beta_j$, choose rationals $q_{k, j} \in (a_j, \beta_j)$ strictly decreasing to $a_j$. By the right-continuity of $I^{(d)}(\cdot, x)$ in each coordinate, $I^{(d)}(q_k, x) \to I^{(d)}(a, x)$ for all $x \in \Xx$.
	  	Therefore, the supremum over  $\mathcal{F}$ is  equal to the supremum over $\mathcal{G}\coloneqq \{ I^{(d)} (q, \cdot),  q\in \Q^d \cap Q\}\subset \Ff$. Hence, the mapping is measurable with respect to the product $\sigma$-algebra, completing the proof.
	 \end{proof}
	
	 \begin{lemma}\label{lem:bound3}  We have
	 	\begin{equation}
	 		(\mu^m_\Xx)^* \Big ( S_m:  \sup_{f\in \Ff}|\E_{\mu_{S_m}}f -\E_{\mu_\Xx} f|  > \frac{C_2\sqrt{\gamma_m}}{C(K, Q)} \Big )\le  4(2m+1) ^d  \exp \Big(-\frac{m C_2 ^2\gamma_m}{8C(K, Q) ^2}\Big).\label{eq:bound3}
	 	\end{equation}
	 \end{lemma}
	 \begin{proof}[Proof of Lemma \ref{lem:bound3}]  The proof  proceeds in four standard steps.
	 	 
	 \underline{Step 1.} {\it Symmetrization.} 
	 We introduce a ``ghost sample" $S_m' = (x_1', \dots, x_m')$ drawn independently from $\mu^m_\Xx$, and  consider Rademacher  variables $\sigma = (\sigma_1, \dots, \sigma_m) \in (\{-1, 1\})^m$   endowed  with the uniform probability measure $\Uu_{\Z_2^m}\in \Pp (\{\pm 1\}^m)$. 	 
	 By Lemma \ref{lem:pm_measurability}, the class $\mathcal{F}$ is sequentially pointwise separable. Because $\Ff$ is uniformly bounded, the Dominated Convergence Theorem ensures that both the empirical expectation $\mathbb{E}_{\mu_{S_m}} f$ and the true expectation $\mathbb{E}_{\mu_\Xx} f$ respect these pointwise limits. Consequently,
	 $$ \sup_{f \in \mathcal{F}} |\mathbb{E}_{\mu_{S_m}} f - \mathbb{E}_{\mu_\Xx} f| = \sup_{f \in \Gg} |\mathbb{E}_{\mu_{S_m}} f - \mathbb{E}_{\mu_\Xx} f|. $$
	Hence, the event 
	 $$ \mathcal{A} \coloneqq \left\{ S_m \in \Xx^m : \sup_{f \in \mathcal{F}} |\mathbb{E}_{\mu_{S_m}} f - \mathbb{E}_{\mu_\Xx} f| > \epsilon_m \right\} $$
	 is measurable.
	
	 If $\Aa$ occurs, there exists some function $f^* \in \mathcal{F}$ (depending on $S_m\in \Aa$) such that 
	 \begin{equation}  
	 	|\E_{\mu_{S_m}} f^* - \E_{\mu_\Xx} f^*| > \epsilon_m. \label{eq:eps}
	 \end{equation}
	 Since $f^* \in \Ff$,  for any $x\in \Xx$
	 \begin{equation} 
	 	f^*(x) \in \{0, 1\} \implies \mathrm{Var}_{\mu_\Xx}(f^*) = \E_{\mu_\Xx}(f^*) - (\E_{\mu_\Xx} f^*)^2 \le \max_{p \in [0,1]} p(1-p) = \frac{1}{4}. \label{eq:varest}
	 \end{equation}
	 
	 Assuming $m\epsilon_m^2 \ge 2$\footnote{If $m\epsilon_m^2 < 2$, the right-hand side of \eqref{eq:bound3} exceeds $1$ and the bound holds trivially; the same convention applies to the assumption $m\tau^2 \ge 8\beta_m^2$ below.} and fixing a sample $S_m \in \mathcal{A}$, Chebyshev's inequality and the bound \eqref{eq:varest} yield the following for the ghost sample $S_m'$:
	 $$ \mu^m_\Xx \left( S_m' \in \Xx^m : |\E_{\mu_{S_m'}} f^* - \E_{\mu_\Xx} f^*| > \frac{\epsilon_m}{2} \right) \le \frac{\mathrm{Var}_\mu(f^*)}{m(\epsilon_m/2)^2} \le \frac{1/4}{m\epsilon_m^2 / 4} = \frac{1}{m\epsilon_m^2} \le \frac{1}{2}. $$
	
	 Consequently, the complementary event $\mathcal{A}_{f^*} \subset \Xx^m$ satisfies:
	 \begin{equation} 
	 	\mu^m_\Xx \left( \mathcal{A}_{f^*} \coloneqq \left\{ S_m' \in \Xx^m : |\E_{\mu_{S_m'}} f^* - \E_{\mu_\Xx} f^*| \le \frac{\epsilon_m}{2} \right\} \right) \ge \frac{1}{2}. \label{eq:complementary}
	 \end{equation} 
	 
	 For any $S_m' \in \mathcal{A}_{f^*}$, taking into account that our fixed $S_m$ satisfies \eqref{eq:eps}, the triangle inequality implies that:
	 \begin{equation} |\E_{\mu_{S_m}} f^* - \E_{ \mu_{S_m'} }f^*| \ge |\E_{\mu_{S_m}} f^* - \E_{\mu_\Xx} f^*| - |\E_{\mu_{S_m'}} f^* - \E_{\mu_\Xx} f^*| > \epsilon_m - \frac{\epsilon_m}{2} = \frac{\epsilon_m}{2}. \label{eq:complementaryc}
	 \end{equation}
	 
	Since $f^* \in \Ff$, we obtain:
	 \begin{equation}
	 	\sup_{f \in \mathcal{F}} |\E_{\mu_{S_m}} f - \E_{\mu_{S_m'}} f| \ge |\E_{\mu_{S_m}} f^* - \E_{\mu_{S_m'}} f^*| > \frac{\epsilon_m}{2}. \label{eq:supest}
	 \end{equation}
	 
	 Since this strict inequality holds for every $S_m' \in \mathcal{A}_{f^*}$, \eqref{eq:complementaryc}  implies that  the event of all $S_m'$ satisfying  \eqref{eq:supest},  given $S_m$,  is a superset of $\mathcal{A}_{f^*}$. Using \eqref{eq:complementary}, we establish a pointwise inequality for our fixed $S_m \in \mathcal{A}$:
	 \begin{equation}
	 	\mu^m_\Xx \left( S_m' \in \Xx^m : \sup_{f \in \mathcal{F}} |\E_{\mu_{S_m}}f - \E_{\mu_{S_m'}} f| > \frac{\epsilon_m}{2} \right) \ge \frac{1}{2}. \label{eq:probsupest}
	 \end{equation}
	 
	 Integrating Inequality \eqref{eq:probsupest} with respect to $\mu^m_\Xx$ over the measurable set $\mathcal{A}$  gives:
	 \begin{align}
	 	\frac{1}{2}\mu^m_\Xx(\mathcal{A}) &= \int_{\mathcal{A}} \frac{1}{2} \, d\mu^m_\Xx(S_m) \nonumber \\
	 	&\stackrel{\eqref{eq:probsupest}}{\le} \int_{\mathcal{A}} \mu^m_\Xx \left( S_m' \in \Xx^m : \sup_{f \in \mathcal{F}} |\E_{\mu_{S_m}} f -\E_{\mu_{S_m'}} f| > \frac{\epsilon_m}{2} \right) d\mu^m_\Xx(S_m) \nonumber \\
	 	&\le \mu^{2m}_\Xx \left( S_m, S_m': \sup_{f \in \mathcal{F}} |\E_{\mu_{S_m}} f - \E_{\mu_{S_m'}} f| > \frac{\epsilon_m}{2} \right). \label{eq:symm}
	 \end{align}

Hence,
	 \begin{equation}
	 	\mu^m_\Xx(\mathcal{A}) \le 2 (\mu^{2m}_\Xx \otimes \Uu_{\Z^m_2})\left( \sup_{f \in \mathcal{F}} \left\vert \frac{1}{m} \sum_{i=1}^m \sigma_i (f(x_i) - f(x_i')) \right\vert > \frac{\epsilon_m}{2} \right). \label{eq:symm1}
	 \end{equation}

	 \underline{Step 2.} {\it Conditioning and Hoeffding's Inequality.}
	 
	 For a fixed function $f \in \mathcal{F}$, define  $c_i (x_i, x_i') \coloneqq f(x_i) - f(x_i')$. Because  $f(x) \in \{0, 1\}$,  we have $c_i(x_i, x_i') \in \{-1, 0, 1\}$. Consequently, $c_i^2 (x_i, x_i') \le 1$. The sum $\frac{1}{m} \sum_{i=1}^m \sigma_i c_i$ is a sum of independent, zero-mean measurable functions. Applying Hoeffding's inequality for Rademacher sums yields:
	 $$\Uu_{\Z^m_2}\left( \left\vert \frac{1}{m} \sum_{i=1}^m \sigma_i c_i \right\vert > \frac{\epsilon_m}{2} \right) \le 2 \exp\left( - \frac{(\epsilon_m/2)^2}{2 \sum_{i=1}^m (c_i/m)^2} \right).$$
	 
	 Because $\sum_{i=1}^m (c_i/m)^2 = \frac{1}{m^2} \sum_{i=1}^m c_i^2 \le \frac{m}{m^2} = \frac{1}{m}$, the bound simplifies exactly to:
	 \begin{equation}
	 	\Uu_{\Z^m_2}\left( \left\vert \frac{1}{m} \sum_{i=1}^m \sigma_i (f(x_i) - f(x_i')) \right\vert > \frac{\epsilon_m}{2} \right) \le 2 \exp\left( - \frac{m\epsilon_m^2}{8} \right). \label{eq:hoeffding}
	 \end{equation}

	\underline{Step 3.} {\it The Union Bound over the Shatter Function.} The supremum in \eqref{eq:symm1} evaluates the class $\mathcal{F}$ over the fixed combined sample $S_{2m}= ( x_1, \ldots, x_m, x_1',\ldots, x_m')$. The number of distinct binary evaluation vectors that $\mathcal{F}$ can produce on $2m$ points is given by the shatter function $\Delta_{2m}(\mathcal{F}, S_{2m})$. Taking the union bound over all distinct projections on the sample, we obtain:
	\begin{equation}	\Uu_{\Z^m_2}\left( \sup_{f \in \mathcal{F}} \left\vert \frac{1}{m} \sum_{i=1}^m \sigma_i (f(x_i) - f(x_i')) \right\vert > \frac{\epsilon_m}{2} \right) \le 2 \Delta_{2m}(\mathcal{F}, S_{2m}) \exp\left( - \frac{m \epsilon_m^2}{8} \right). \label{eq:unionbound}
	\end{equation}
	
	\underline{Step 4.} {\it Applying the VC-Index Bounds.} The class $\mathcal{F} = \{I^{(d)}(a, \cdot) : a \in Q\}$ corresponds to the collection of lower-left orthants, which forms a VC-class of index $V = d+1$ \cite[Example 2.6.1]{VW1996}. By the Sauer-Shelah lemma,\footnote{The combinatorial bound on the shatter function, frequently referred to as the Sauer-Shelah lemma, was discovered independently by Vapnik and Chervonenkis \cite{VC1971}, Sauer \cite{Sauer1972}, and Shelah \cite{Shelah1972}.} the shatter function on $2m$ points is bounded by a polynomial of degree $V-1$:
	\begin{equation}
		\max_{S_{2m}\in \Xx^{2m}} \Delta_{2m}(\mathcal{F}, S_{2m}) \le (2m+1)^d.\label{eq:vcindex}
	\end{equation}
	Substituting \eqref{eq:vcindex} into \eqref{eq:unionbound},  and taking the expectation over $S_{2m}$ leaves the bound unchanged. Finally, multiplying by the leading constant $2$ from the symmetrization step \eqref{eq:symm1} yields the  bound:
	\begin{align*}
		(\mu^m_\Xx)\Big(S_m: \sup_{f \in \mathcal{F}} \vert\E_{\mu_{S_m}} f - \E_{\mu_\Xx} f\vert > \epsilon_m \Big) &\le 2 \times 2(2m+1)^d \exp\left( - \frac{m \epsilon_m^2}{8} \right)\\
		& = 4(2m+1)^d \exp\left( - \frac{m \epsilon_m^2}{8} \right).
	\end{align*}
	Setting the threshold $\epsilon_m \coloneqq \frac{C_2 \sqrt{\gamma_m}}{C(K, Q)}$ yields the required exponential bound \eqref{eq:bound3}, completing the proof of Lemma \ref{lem:bound3}.

	\end{proof}

	{\it Completion of the proof of Theorem \ref{thm:genvapnik}.}  From   \eqref{eq:bound3a}  and Lemma \ref{lem:bound3}   we obtain the bound for the second term in the RHS of \eqref{eq:master_decomp}
	\begin{equation}\label{eq:masterb2}
		(\mu ^m)^*( S_m \in \Zz^m:\| A_{S_m}- A_{\mu_\Xx}\|_{\mathrm{op}}>  C_2\sqrt{\gamma_m} )\le   4 (2m+1)  ^d \exp  \big (  - \frac{ m C_2 ^2  \gamma_m}{ 8 C(K, Q)^2}\big ).
	\end{equation}
	It remains to bound the first term in the RHS of \eqref{eq:master_decomp}. Recall that $S_m \in \mathcal{Z}^m$. To bound $\Vert B_{S_m}^g - B_\mu^g \Vert_{L^2(\nu)}$ for an unbounded target $g$, we employ a truncation argument standard in the nonparametric regression literature (see, e.g., Gy\"orfi et al. \cite[Chapter 11]{Gyorfi2002}). We introduce a sequence of truncation thresholds $\beta_m > 0$ and define
	$$g_{\beta_m}(y) \coloneqq g(y) \cdot \mathbf{1}_{\{\vert g(y)\vert \le \beta_m\}}.$$
	By the triangle inequality:
	\begin{equation}
		\Vert B_{S_m}^g - B_\mu^g \Vert_{L^2(\nu)} \le \Vert B_{S_m}^g - B_{S_m}^{g_{\beta_m}} \Vert_{L^2(\nu)} + \Vert B_{S_m}^{g_{\beta_m}} - B_\mu^{g_{\beta_m}} \Vert_{L^2(\nu)} + \Vert B_\mu^{g_{\beta_m}} - B_\mu^g \Vert_{L^2(\nu)}.\label{eq:trunc_decomp}
	\end{equation}

	Because the $L^2(\nu)$-norm is bounded by the sup norm, we have:
	\begin{equation}
		\Vert B_\mu^{g_{\beta_m}} - B_\mu^g \Vert_{L^2(\nu)} \le \sup_{a \in Q} \int_{\mathcal{X} \times \mathcal{Y}} I^{(d)}(a, x) \vert g(y)\vert \mathbf{1}_{\{\vert g(y)\vert > \beta_m\}} \, d\mu \le \frac{M}{\beta_m}.\label{eq:poperr}
	\end{equation}

	Applying Markov's inequality yields:
	\begin{equation}
		\mu^m\left( \Vert B_{S_m}^g - B_{S_m}^{g_{\beta_m}} \Vert_{L^2(\nu)} > \tau \right) \le \frac{1}{\tau} \E_{\mu}\left[ \frac{1}{m} \sum_{i=1}^m \vert g(y_i)\vert \mathbf{1}_{\{\vert g(y_i)\vert > \beta_m\}} \right] \le \frac{M}{\beta_m \tau}.\label{eq:markov_empi}
	\end{equation}
	Here  the last  inequalities in \eqref{eq:poperr}  and \eqref{eq:markov_empi}  are obtained by  using  the   inequality  $| g (y) |\le  | g (y) | ^2/\beta_m$ if $|g  (y) | > \beta_m$.

	For the middle term of \eqref{eq:trunc_decomp}, let $\tau \coloneqq \frac{C_1 \sqrt{\gamma_m}}{3}$. Because the truncated function class $\{g_{\beta_m} I^{(d)}(a, \cdot) : a \in Q\}\subset \R^{\Yy \times \Xx}$ inherits  the sequential pointwise separability from the indicator class $\mathcal{F}$, the event
	$$ \mathcal{A}_\tau \coloneqq \left\{ S_m \in \mathcal{Z}^m : \Vert B_{S_m}^{g_{\beta_m}} - B_\mu^{g_{\beta_m}} \Vert_{L^2(\nu)} > \tau \right\} $$
	is  measurable. Assuming $m\tau^2 \ge 8\beta_m^2$ and fixing $S_m \in \mathcal{A}_\tau$, there exists a parameter $a^*$ such that the empirical deviation strictly exceeds $\tau$. The variance of the bounded function $g_{\beta_m} I^{(d)}(a^*, \cdot)$ satisfies $\mathrm{Var}_\mu \le \beta_m^2$. Applying Chebyshev's inequality to this fixed function over the ghost sample $S_m'$, we obtain:
	
\begin{align*} \mu^m \Big( S_m' \in \mathcal{Z}^m : |\E_{\mu_{S_m'}} (g_{\beta_m} I^{(d)}(a^*, \cdot)) - \E_\mu(g_{\beta_m} I^{(d)}(a^*, \cdot))|& > \frac{\tau}{2} \Big) \le \frac{\beta_m^2}{m(\tau/2)^2}\\
	& = \frac{4\beta_m^2}{m\tau^2} \le \frac{1}{2}.
\end{align*}

	Following the same symmetrization logic from Lemma \ref{lem:bound3}, integrating over the measurable set $\mathcal{A}_\tau$ and applying independent Rademacher variables $\sigma_i$ yields:
	\begin{align}
		\mu^m(\mathcal{A}_\tau) &\le 2 (\mu^{2m} \otimes \Uu_{\Z^m_2})\Big( \sup_{a \in Q} \Big\vert \frac{1}{m} \sum_{i=1}^m \sigma_i \big( g_{\beta_m}(y_i) I^{(d)}(a, x_i) \nonumber\\
		& - g_{\beta_m}(y_i') I^{(d)}(a, x_i') \big) \Big\vert > \frac{\tau}{2} \Big). \label{eq:symm_trunc}
	\end{align}
	
	Defining the symmetrized summand 
	$$c_i(a) \coloneqq g_{\beta_m}(y_i) I^{(d)}(a, x_i) - g_{\beta_m}(y_i') I^{(d)}(a, x_i'),$$
	 we  have $|c_i(a)| \le 2\beta_m$.  Let $S_{2m}^{\mathcal{X}} \coloneqq (x_1, \dots, x_m, x_1', \dots, x_m') \in \mathcal{X}^{2m}$ be the projection of the double sample onto $\mathcal{X}$. For any fixed realization of the double sample $S_{2m}$, the values $y_i$ and $y_i'$ are strictly determined data points. Furthermore, because the supremum is taken over $a \in Q$, the terms $g_{\beta_m}(y_i)$ and $g_{\beta_m}(y_i')$ do not depend on $a$ and therefore act as fixed, real-valued scalar weights. Consequently, as $a$ varies across $Q$, the value of the vector $(c_1(a), \dots, c_m(a))$ is completely and uniquely determined by the binary evaluation vector of the indicator class $\mathcal{F}$ on $S_{2m}^{\mathcal{X}}$:
	 $$ \Big(I^{(d)}(a, x_1), \dots, I^{(d)}(a, x_m), I^{(d)}(a, x_1'), \dots, I^{(d)}(a, x_m') \Big). $$ 
	 If two parameters $a_1, a_2 \in Q$ yield the same binary vector for $\mathcal{F}$ on $S_{2m}^{\mathcal{X}}$, they guarantee the exact same real-valued vector for the summands, i.e., $c_i(a_1) = c_i(a_2)$ for all $i \in [1, m]$. 
	 
	 Therefore, the number of distinct realizations of the supremum inside the probability over the induced class $\tilde{\mathcal{F}} \coloneqq \{g_{\beta_m}(\cdot) I^{(d)}(a, \cdot) : a \in Q\}$ on $S_{2m}$ is strictly bounded by the shatter function of the original VC-class $\mathcal{F}$ on $S_{2m}^{\mathcal{X}}$. By the Sauer-Shelah lemma, see \eqref{eq:vcindex}, this is bounded by $\Delta_{2m}(\mathcal{F}, S_{2m}^{\mathcal{X}}) \le (2m+1)^d$.

	 Conditioning on the double sample $S_{2m}$ and applying a union bound over the VC-class $\mathcal{F}$ via the Sauer-Shelah lemma,  see \eqref{eq:vcindex}, the Hoeffding bound becomes: 
	\begin{align}
		\mu^m(\Aa_\tau) &\le 2 \, \E_{S_{2m}\sim \mu^{2m}} \Big[ 2 \Delta_{2m}(\Ff, S_{2m}^\Xx) \exp\big( - \frac{(\tau/2)^2}{2 \sum_{i=1}^m (2\beta_m/m)^2} \big) \Big] \nonumber \\
		&\le 4 (2m+1)^d \exp\big( - \frac{m \tau^2}{32 \beta_m^2} \big). \label{eq:empierror}
	\end{align}

	To assemble the total error, we apply a union bound to \eqref{eq:trunc_decomp} with threshold $\tau \coloneqq \frac{C_1 \sqrt{\gamma_m}}{3}$. Because the population error \eqref{eq:poperr} is deterministic, its probability of exceeding $\tau$ is exactly the indicator function $\mathbf{1}_{\{ M  > \beta_m\tau \}}$. Substituting \eqref{eq:markov_empi}, \eqref{eq:empierror}, and the deterministic indicator into the union bound, and combining with the operator error from Lemma \ref{lem:bound3}, the master decomposition \eqref{eq:master_decomp} evaluates to:  		
	\begin{align}
		(\mu^m)^*\big( S_m: \Vert f_{S_m} - r_g^\mu \Vert_{\mathcal{H}_K}^2 > \eps \big) 
		&\le \frac{3 M}{C_1\beta_m \sqrt{\gamma_m}} + \mathbf{1}_{\{ 3M  > C_1 \beta_m\sqrt{\gamma_m} \}} \nonumber \\
		&\quad + 4(2m+1)^d \exp\left( - \frac{m C_1^2 \gamma_m}{288 \beta_m^2} \right) \nonumber \\
		&\quad + 4(2m+1)^d \exp\left( - \frac{m C_2^2 \gamma_m}{8 C(K, Q)^2} \right). \label{eq:final_assembly}
	\end{align} 
	By setting the truncation sequence $\beta_m$ such that 
	$$\lim_{m \to \infty} \beta_m \sqrt{\gamma_m} = \infty,$$ the indicator function in \eqref{eq:final_assembly} eventually evaluates to $0$, and the Markov bound decays to $0$. Coupling this with the rate condition 
	$$\lim_{m \to \infty} \frac{m \gamma_m}{ \beta_m^2\log m }= \infty$$
	 ensures the exponential terms vanish, establishing \eqref{eq:vapnik_bound}.
\end{proof}

\begin{remark}[Existence of $\gamma_m$]  By setting  $\gamma_m = m ^{-a}$, $\beta_m = m^b$ one verifies that  the  condition  that  $\lim_{m \to \infty}\gamma_m  = 0$ is equivalent to  $a >0$, the condition  that $\lim_{m \to \infty}\beta_m \sqrt{\gamma_m} = \infty $ is equivalent  to  $\frac{a}{2} < b $ and the  condition that  $\lim_{m \to \infty}\frac{m\gamma_m}{\beta_m^2 \log m}= \infty$  is equivalent  to $a + 2b <1$. Thus  letting $b  = 1/4$ we can take  $0 < a < \frac{1}{2}$.
\end{remark}

\begin{remark}[Comparison with Vapnik-Izmailov \cite{VI2015}]\label{rem:VI}
	Note that while Vapnik and Izmailov \cite[Section 2.4]{VI2015} formulated the estimation of the regression function
	$$ r(x) = \int y p(y|x) dy = \int y \frac{p(x, y)}{p(x)} dy $$
	as a stochastic ill-posed problem, their framework implicitly assumes the existence of the conditional density $p(y|x)$ (and consequently, absolute continuity with respect to the Lebesgue measure). Furthermore, they did not propose a Hilbert space extension of Vapnik's theorem \cite[Theorem 7.2]{Vapnik1998} to rigorously solve this problem.
\end{remark}
\begin{remark}[Comparison with other methods]\label{rem:genvapnik}
	For foundational literature detailing the non-parametric estimation of conditional expectations, we refer the reader to Gy\"orfi et al. \cite{Gyorfi2002}. Standard distribution-free regression methods (e.g., Nadaraya-Watson kernel smoothing, $k$-nearest neighbors, or least-squares partitioning) typically estimate the conditional expectation by performing local averaging to directly minimize an empirical $L^2$ risk. 
	
	In contrast, the estimates  in  \eqref{eq:master_decomp}, \eqref{eq:masterb2}, \eqref{eq:final_assembly} constitute a  new  approach  based  on  solving  a  stochastic ill-posed problem  and Vapnik-Chervonenkis empirical process theory.
 By \eqref{eq:lipsmooth}, this estimation inherently provides pointwise-evaluable, and  $C^r$-guarantees for any nonnegative  integer $r$. 
\end{remark}

In  the following theorem  we provide a  sufficient condition  for the uniform consistency of the  algorithm of Theorem  \ref{thm:genvapnik} with respect to $C^r$-distance.

\begin{theorem}\label{thm:uniconsistency}   Assume the statistical learning model above and the rate conditions
	\eqref{eq:gamma_rate}.   Suppose that  the following conditions  (U1)  and (U2) are satisfied  for $\Pp_0 \subset \Pp_K  (\Zz)$:\\
(U1)  There exist    $C_0, M \in\R_{+}$  such that
 $$\sup_{\mu\in \Pp_0} \| r_g ^\mu\| _{\Hh_K}\le C_0, \qquad \sup_{ \mu \in \Pp_0} \int_\Yy |g|^2 d\mu_\Yy \le M. $$

 (U2) Let
\[
\mathfrak M_{\mathcal X}
=\{\mu_{\mathcal X}:\mu\in\mathcal P_0\}
\subset\mathcal P(\mathcal X).
\]
Assume that every measure belonging to the closure of
$\mathfrak M_{\mathcal X}$ in the weak topology of
$\mathcal P(\mathcal X)$ has support $\mathcal X$.

Then  for  every   nonnegative integer  $r$ and $ \eps >0$   we  have
\begin{equation}\label{eq:UCr}
	\lim_{m \to \infty} \sup_{\mu  \in \Pp_0}  (\mu^m)^* \{  S_m : \| f_{S_m} - r^\mu_g\| _{C^r (\Xx)} \ge  \eps \} = 0.
\end{equation}
\end{theorem}
\begin{remark}\label{rem:U2}   Sufficient   special cases of   (U2)   are:\\
 1)  $\mu_\Xx$ is fixed and has   full support;\\
 2)  $\mathfrak M_{\mathcal X}$  is weakly closed and  each  $\mu_\Xx$  has full support;\\
 3) for some fixed  full-support  probability measure $\rho$  there exists $ c> 0$ such that
 $$ \mu_\Xx \ge c \rho, \qquad \forall  \mu \in \Pp_0 .$$
\end{remark}
\begin{proof}[Proof of Theorem  \ref{thm:uniconsistency}]  Fix arbitrary constants $C_1, C_2>0$  and  as in \eqref{eq:cm} define
\begin{equation}
\Cc_{m, \mu} \coloneqq \{ S_m 
\in \Zz^m: \| B^g_{S_m} - B^g_\mu\| \le C_1 \sqrt{\gamma_m}, \;  \| A_{S_m} - A_{\mu_\Xx}\|_{\mathrm{op}}  \le  C_2 \sqrt{\gamma_m} \}. \label{eq:cmm}
\end{equation}
Set 
$$b \coloneqq C_1  + C_2 C_0, \qquad  R_0 \coloneqq \sqrt{ C_0^2 + b^2}.  $$
By \eqref{eq:fmbounded}  and  (U1)   we have
\begin{equation} \| f_{S_m} \| _{\Hh_K} \le  R_0, \qquad   	\| f_{S_m}  - r^\mu _g\|  \le R_0 + C_0,   \qquad  \forall  S_m \in \Cc_{m, \mu}   \label{eq:uni3} 
\end{equation}

By \eqref{eq:dist22},  taking into  account \eqref{eq:uni3}, we have
\begin{equation}
	\| A_{\mu_\Xx} ( f_{S_m} - r^\mu_g)\|_{L^2 (\nu)} \le [  C_1  + (1+C_2)R_0]\sqrt{\gamma_m},  \;  \forall  S_m \in \Cc_{m, \mu}.\label{eq:un4}
\end{equation}
 Condition (U2) implies  the following uniform  stability statement. For  any $R, \eps >0$ there exists  $ \delta  >0$ such that  for all $\mu \in \Pp_0$ we have
 \begin{equation}
 	\| u \|_{\Hh_K} \le R, \qquad \| A_{\mu_\Xx} u\|_{L^2 (Q, \nu)} \le \delta  \LRA  \| u \|_{C^r} < \eps. \label{eq:un2}
 \end{equation}
Indeed,   otherwise   one could find  $\mu_j$ and $u_j$  such that 
\begin{equation*}
\| u_j \|_{\Hh_K} \le R, \qquad \| A_{\mu_{j,\Xx}} u_j\|_{L^2 (\nu)}\to  0, \qquad  \|u_j\|_{C^r} \ge  \eps. 
\end{equation*}
After passing to   subsequences,  taking into account the compactness  of the inclusion 
$\Hh_K \to   C^r$, we have
\begin{equation}
	 u_j   \rightharpoonup u \text{ in } \Hh_K , \qquad u_j \to u \text{ in }  C^r, \qquad  \mu_{j, \Xx} \LRA  \lambda. \label{eq:un2n} 
\end{equation}
From \eqref{eq:un2n} we conclude  that
$$\| A_{\mu_{j,\Xx}} (u_j -u)\|_{L^2 (\nu)}\le \| u_j- u\|_\infty  \to 0.$$
For  fixed  $u$ the signed measure  $u\mu_{j, \Xx}$ converges weakly to  $u\lambda$.
For each $i\in\{1,\ldots,d\}$, the set
\[
D_i=\{t\in\R:\lambda(\{x\in\Xx:x_i=t\})>0\}
\]
is at most countable. Hence, for Lebesgue-almost every $a\in Q$,
the lower orthant $(-\infty,a]$ is a $\lambda$-continuity set.
Since $\nu$ is normalized Lebesgue measure on $Q$, this holds
$\nu$-almost everywhere. Therefore, the weak convergence
$u\mu_{j,\Xx}\Rightarrow u\lambda$ implies
\[
(A_{\mu_{j,\Xx}}u)(a)\longrightarrow(A_\lambda u)(a)
\]
for $\nu$-almost every $a\in Q$. Dominated convergence then gives
$A_{\mu_{j,\Xx}}u\to A_\lambda u$ in $L^2(Q,\nu)$. Hence $A_\lambda u=0$. By (U2),
$\operatorname{supp}\lambda=\Xx$, and therefore
Lemma \ref{lem:regillp}(3) gives $u=0$. Since
$u_j\to u$ in $C^r(\Xx)$, this contradicts
$\|u_j\|_{C^r}\ge\eps$.
 This proves \eqref{eq:un2}.

Using the conditional-stability statement \eqref{eq:un2},  and noting  that  $\gamma_m \to 0$,   we derive  from \eqref{eq:un4}
$$\|f_{S_m}  -  r^\mu_g \| _{C^r}  < \eps \qquad \forall  S_m \in \Cc_{m, \mu}.$$
 for all  sufficiently  large  $m$, with  threshold  $\gamma_0$  independent  of $\mu$.  
 The concentration estimates used in deriving
 \eqref{eq:final_assembly} show, uniformly in
 $\mu\in\Pp_0$, that
 \[
 (\mu^m)^*(\Cc_{m,\mu}^c)\longrightarrow0.
 \]
 For all sufficiently large $m$,
 \[
 \left\{
 \|f_{S_m}-r_g^\mu\|_{C^r}\ge\eps
 \right\}
 \subseteq\Cc_{m,\mu}^c.
 \]
 Consequently,
 \[
 \sup_{\mu\in\Pp_0}
 (\mu^m)^*
 \left\{
 \|f_{S_m}-r_g^\mu\|_{C^r}\ge\eps
 \right\}
 \le
 \sup_{\mu\in\Pp_0}
 (\mu^m)^*(\Cc_{m,\mu}^c)
 \longrightarrow0,
 \]

 which proves Theorem \ref{thm:uniconsistency}.
\end{proof}
   For further applications   of  a variant  of  Theorem  \ref{thm:cvapnikhl}  and  Vapnik's theorem \cite[Theorem 7.3]{Vapnik1998}, refer  to \cite{LMPW2026}.

\appendix
\section{Proof of   Lemma \ref{lem:HK}} \label{sec:HK}
In this  Appendix  we  first recall  the concept of  Hardy-Krause variation  and the   generalized Koksma-Hlawka inequality due to Aistleitner-Dick \cite{AD2015} which  we shall use  in the proof of Lemma \ref{lem:HK}. Then we give  a  proof  of Lemma \ref{lem:HK}.

Let
\[
Q=\prod_{j=1}^d[\alpha_j,\beta_j]\subset\R^d.
\]
We first recall the definition of the {\it Vitali variation}. Let
\[
R=\prod_{j=1}^d[c_j,d_j]\subset\R^d
\]
and let $h:R\to\R$. For $s,t\in R$ with $s_j<t_j$ for
every $j$, define the mixed difference of $h$ over $[s,t]$ by
\[
\Delta(h;[s,t])
\coloneqq
\sum_{\varepsilon\in\{0,1\}^d}
(-1)^{d-|\varepsilon|}h(z_\varepsilon),
\]
where
\[
(z_\varepsilon)_j
=
\begin{cases}
	s_j,&\varepsilon_j=0,\\
	t_j,&\varepsilon_j=1.
\end{cases}
\]
For grid partitions
\[
c_j=t_{j,0}<t_{j,1}<\cdots<t_{j,N_j}=d_j,
\qquad 1\le j\le d,
\]
let $R_{\mathbf i}$ denote the corresponding subrectangles. The
$d$-dimensional Vitali variation of $h$ on $R$ is
\[
V^{(d)}(h,R)
\coloneqq
\sup_{\mathcal P}
\sum_{\mathbf i}
\left|\Delta(h;R_{\mathbf i})\right|,
\]
where the supremum is taken over all such grid partitions
$\mathcal P$.

For a function $g:Q\to\R$ and every nonempty
$u\subseteq\{1,\ldots,d\}$, let
\begin{equation}
	Q_u=\prod_{j\in u}[\alpha_j,\beta_j],
	\qquad
	g_u(x_u)\coloneqq g(x_u,\beta_{-u}).\label{eq:qugu}
\end{equation}
The {\it Hardy--Krause variation} of $g$, anchored at the upper corner
$\beta=(\beta_1,\ldots,\beta_d)$, is
\[
V_{HK}(g,Q)
\coloneqq
\sum_{\emptyset\ne u\subseteq\{1,\ldots,d\}}
V^{(|u|)}(g_u,Q_u).
\]

We shall also use the following measure-theoretic description. Let
\[
R^{\mathrm{ho}}
\coloneqq
\prod_{j=1}^d[c_j,d_j).
\]
Suppose that the mixed increments of \(h\) determine a finite signed
Borel measure $\lambda_h$ on $R^{\mathrm{ho}}$ by
\[
\lambda_h\left(\prod_{j=1}^d[s_j,t_j)\right)
=
\Delta(h;[s,t]).
\]
Then
\[
V^{(d)}(h,R)=|\lambda_h|(R^{\mathrm{ho}}).
\]
On $\operatorname{int}R$, the measure $\lambda_h$ agrees with the
mixed distributional derivative of $h$. The half-open formulation
also records possible jumps at the lower boundary of $R$.

In particular, if $g\in C^d(Q)$, then
\[
V_{HK}(g,Q)
=
\sum_{\emptyset\ne u\subseteq\{1,\ldots,d\}}
\int_{Q_u}
\left|
\frac{\partial^{|u|}g(x_u,\beta_{-u})}{\partial x_u}
\right|\,dx_u.
\]

\begin{theorem}[Generalized Koksma--Hlawka inequality]
	\label{thm:AD2}
	Let $\mu,\nu\in\Pp(Q)$, and define
	\[
	F_\mu(a)=\mu([\alpha,a]),
	\qquad
	F_\nu(a)=\nu([\alpha,a]).
	\]  If $g:Q\to\R$ is measurable and has bounded
	Hardy--Krause variation, then
	\[
	\left|
	\int_Q g\,d\mu-\int_Q g\,d\nu
	\right|
	\le
	V_{HK}(g,Q)
	\sup_{a\in Q}|F_\mu(a)- F_\nu(a)|.
	\]
\end{theorem}
This two-measure form follows from
\cite[Theorem 1]{AD2015} by a probabilistic approximation
argument. Let $X_1,X_2,\ldots$ be i.i.d. with distribution $\mu$ and set
\[
\mu_N\coloneqq\frac1N\sum_{i=1}^N\delta_{X_i}.
\]
By the multivariate Glivenko--Cantelli theorem \ref{thm:GC} and the strong law
of large numbers, there exists a realization such that
\[
\sup_{a\in Q}
|\mu_N([\alpha,a])-\mu([\alpha,a])|\longrightarrow0
\]
and
\[
\int_Q g\,d\mu_N\longrightarrow\int_Q g\,d\mu.
\]
Applying \cite[Theorem 1]{AD2015} to this realization, with $\nu$
as the reference measure, and then letting $N\to\infty$ yields the
asserted inequality.

\begin{proof}[Proof  of Lemma \ref{lem:HK}]   Recall that $   \Xx \subset Q\subset \R^d$.  Let
	\[
	C^d(Q)
	\coloneqq
	\left\{
	g\in C^d(\operatorname{int}Q):
	\partial^\alpha g\text{ extends continuously to }Q
	\text{ for every }|\alpha|\le d
	\right\}.
	\]
	For such a function, the same symbol $\partial^\alpha g$ denotes its
	continuous extension to $Q$, and
	\[
	\|g\|_{C^d(Q)}
	\coloneqq
	\max_{|\alpha|\le d}\sup_{x\in Q}
	|\partial^\alpha g(x)|.
	\]
	Equivalently, $C^d(Q)$ consists of the restrictions to $Q$ of
	$C^d$-functions defined on an open neighborhood of $Q$.
	Define
	$$ C^d_b (\R^d) \coloneqq \Big \{ F\in  C^d (\R^d): \| F\| _{C^d_b} \coloneqq \max_{|\alpha|\le d}\sup_{\R^d} |\p ^\alpha F| < \infty \Big\}.$$  By Fefferman's bounded linear extension theorem
	\cite[Theorem 1]{Fefferman2007},  there exists  a bounded linear operator
	\[
	\mathcal E_d:C^d(\mathcal X)\longrightarrow C_b^d(\R^d)
	\]
	such that
	\[
	(\mathcal E_df)|_{\mathcal X}=f
	\]
	and
	\begin{equation}\label{eq:extensionCm}
		\|\mathcal E_df\|_{C_b^d(\R^d)}
		\le
		C_{\mathrm{ext}}(\mathcal X)
		\|f\|_{C^d(\mathcal X)}.
	\end{equation}
	Here $C^d(\mathcal X)$ is identified, with an equivalent norm, with
	the $C^d(\R^d)$-trace space on $\mathcal X$.
	(For a compact embedded smooth manifold, possibly with smooth
	boundary, the intrinsic $C^d(\Xx)$-norm is equivalent to the
	$C^d(\R^d)$-trace norm. This follows by constructing local
	extensions in finitely many interior and boundary charts and
	patching them with a partition of unity.)
	
	For $f\in\mathcal H_K$, put $\widetilde f=(\mathcal E_df)|_Q$ and define
	\[
	I_f^a(x)\coloneqq
	I^{(d)}(a,x)\widetilde f(x),
	\qquad a,x\in Q.
	\]
	
	Then
	\begin{equation}\label{eq:extensionCmQ}
		\|\widetilde f\|_{C^d(Q)}
		\le
		C_{\mathrm{ext}}(\mathcal X)
		\|f\|_{C^d(\mathcal X)}.
	\end{equation}
	
	Fix a nonempty $u\subseteq\{1,\ldots,d\}$. If
	$a_j<\beta_j$ for some $j\notin u$, then by \eqref{eq:qugu}
	\[
	(I_f^a)_u(x_u)
	=
	I_f^a(x_u,\beta_{-u})=0.
	\]
	Otherwise, putting
	\[
	\phi_u(x_u)\coloneqq
	\widetilde f(x_u,\beta_{-u}),
	\]
	we have
	\[
	(I_f^a)_u(x_u)
	=
	\prod_{j\in u}\mathbf 1_{\{x_j\le a_j\}}\,
	\phi_u(x_u).
	\]
	Let
	\[
	J_{u,a}\coloneqq\{j\in u:a_j<\beta_j\}.
	\]
	Define the finite signed Borel measure $\lambda_{u,a}$ on
	\[
	Q_u^{\mathrm{ho}}
	=
	\prod_{j\in u}[\alpha_j,\beta_j)
	\]
	as the sum, indexed by $v\subseteq J_{u,a}$, of the measures
	\[
	(-1)^{|v|}
	\mathbf 1_{\prod_{j\in u\setminus v}[\alpha_j,a_j]}
	(x_{u\setminus v})
	\,
	\partial_{u\setminus v}\phi_u(x_{u\setminus v},a_v)
	\,dx_{u\setminus v}\otimes\delta_{a_v},
	\]
	with the usual interpretation when $v=\varnothing$ or $v=u$.
	The one-sided distributional product rule, or equivalently a direct
	calculation of mixed increments, gives
	\[
	\lambda_{u,a}
	\left(\prod_{j\in u}[s_j,t_j)\right)
	=
	\Delta\bigl((I_f^a)_u;[s,t]\bigr).
	\]
	Consequently,
	\[
	V^{(|u|)}((I_f^a)_u,Q_u)
	=
	|\lambda_{u,a}|(Q_u^{\mathrm{ho}}).
	\]
	
	Writing $L_j=\beta_j-\alpha_j$, we consequently obtain
	\begin{align*}
		V^{(|u|)}((I_f^a)_u,Q_u)
		&=
		|\lambda_{u,a}|(Q_u^{\mathrm{ho}})\\
		&\le
		\sum_{v\subseteq J_{u,a}}
		\prod_{j\in u\setminus v}L_j\,
		\|\widetilde f\|_{C^d(Q)}\\
		&\le
		\prod_{j\in u}(1+L_j)\,
		\|\widetilde f\|_{C^d(Q)}.
	\end{align*}
	Thus the estimate is uniform in $a\in Q$. Summing over all
	nonempty $u$, we obtain
	\begin{equation}\label{eq:HKCm}
		V_{HK}(I_f^a,Q)
		\le
		C(Q,d)\|\widetilde f\|_{C^d(Q)},
		\qquad a\in Q,
	\end{equation}
	where, for example, one may take
	\[
	C(Q,d)
	=
	\sum_{\emptyset\ne u\subseteq\{1,\ldots,d\}}
	\prod_{j\in u}(1+\beta_j-\alpha_j).
	\]
	We derive Inequality \eqref{eq:HK}  from  Theorem \ref{thm:AD2}  and \eqref{eq:HKCm}, taking into account  \eqref{eq:lipsmooth},  immediately.
\end{proof}

\end{document}